\documentclass[a4,12pt]{amsart}

\usepackage{amssymb,amstext,amsmath,amscd,amsthm,amsfonts,enumerate,latexsym}
\usepackage{color}
\usepackage[dvipdfmx]{graphicx}
\usepackage[all]{xy}
\usepackage{extarrows}

\usepackage{ytableau}

\usepackage{array}
\usepackage{xcolor}
\usepackage[leqno]{amsmath}
\usepackage{mathptmx}

\newtheorem{thm}{Theorem}[section]
\newtheorem{prop}[thm]{Proposition}
\newtheorem{cor}[thm]{Corollary}
\newtheorem{lem}[thm]{Lemma}

\theoremstyle{definition}
\newtheorem{rem}[thm]{Remark}
\newtheorem{defn}[thm]{Definition}
\newtheorem{ex}[thm]{Example}

\newtheorem{setting}[thm]{Setting}

\numberwithin{equation}{section}

\newtheorem*{ac}{Acknowledgments}

\numberwithin{equation}{thm}

\def\Im{\operatorname{Im}}

\def\Hom{\operatorname{Hom}}

\def\Proj{\operatorname{Proj}}

\def\Coker{\mathrm{Coker}}

\def\rank{\mathrm{rank}}

\def\m{\mathfrak m}

\def\H{\mathrm{H}}

\newcommand{\rmH}{\mathrm{H}}
\newcommand{\rmI}{\mathrm{I}}

\newcommand{\rmQ}{\mathrm{Q}}

\newcommand{\calF}{\mathcal{F}}

\newcommand{\calR}{\mathcal{R}}
\newcommand{\calS}{\mathcal{S}}

\newcommand{\fka}{\mathfrak{a}}

\newcommand{\fkM}{\mathfrak{M}}

\def\depth{\operatorname{depth}}
\def\Supp{\operatorname{Supp}}

\def\Ass{\operatorname{Ass}}
\def\height{\mathrm{ht}}

\def\Spec{\operatorname{Spec}}
\def\Sym{\mathrm{Sym}}

\def\ord{\operatorname{ord}}
\def\gr{\mbox{\rm gr}}
\def\Fitt{\mathrm{Fitt}}

\def\M{{\mathcal M}}

\newcommand{\rmbr}{\operatorname{br}}

\begin{document}

\setlength{\baselineskip}{15pt}
\title[On Ratliff--Rush closure of modules]{On Ratliff--Rush closure of modules}
\author{Naoki Endo}
\address{Global Education Center \\
 Waseda University \\
  1-6-1 Nishi-Waseda, Shinjuku-ku \\
 Tokyo 169-8050 \\ 
 Japan}
\email{naoki.taniguchi@aoni.waseda.jp}
\urladdr{http://www.aoni.waseda.jp/naoki.taniguchi/}


\begin{abstract} 
In this paper, we introduce the notion of Ratliff--Rush closure of modules and explore whether the condition of the Ratliff--Rush closure coincides with the integral closure. The main result characterizes the condition in terms of the normality of the projective scheme of the Rees algebra. In conclusion, we shall give a criterion for the Buchsbaum Rees algebras. 
\end{abstract}


\maketitle


\section{Introduction}

This paper investigates the Ratliff--Rush closure and the Buchsbaum property for the Rees algebras of modules. For an arbitrary ideal $I$ in a commutative Noetherian ring $A$, we set 
$$
\widetilde{I} = \displaystyle\bigcup_{\ell \ge 0} \left[ I^{\ell +1} :_A I^{\ell}\right]
$$
and name it {\it the Ratliff--Rush closure of $I$}, which forms an ideal of $A$, containing $I$. 
In 1978, L. J. Ratliff and D. E. Rush investigated the ideal $\widetilde{I}$, and they proved that $(\widetilde{I})^n = I^n$ for every $n \gg 0$. In addition, if $J$ is an ideal of $A$ such that $J^n = I^n$ for every $n \gg 0$, then $J \subseteq \widetilde{I}$. Therefore, $\widetilde{I}$ is the largest ideal of $A$ satisfying $(\widetilde{I})^n = I^n$ for a sufficiently large integer $n \gg 0$, and hence $\widetilde{\widetilde{I}\hspace{0.2em}} = \widetilde{I}$. The products of the Ratliff--Rush closures are contained in the Ratliff--Rush closure of the products of ideals. Moreover, if $I$ possesses a positive grade, then $I$ is a reduction of its Ratliff--Rush closure $\widetilde{I}$; in other words, the integral closure $\overline{I}$ of $I$ contains $\widetilde{I}$. One can consult \cite{Mac, RR} for basic properties of Ratliff--Rush closure of ideals.

In 2005, S. Goto and N. Matsuoka focused on the difference between $\widetilde{I}$ and $\overline{I}$, and explored the question of when does the Ratliff--Rush closure coincide with the integral closure. Over a two-dimensional regular local ring $A$, they provided a characterization of the equality $\widetilde{I} = \overline{I}$ in terms of the condition that the Rees algebra
$$
\calR(I) = A[It] =\sum_{i\ge 0}I^it^i \subseteq A[t]
$$
of the ideal $I$ is locally normal on $\Spec\calR(I)\setminus\{\fkM\}$, where $t$ denotes an indeterminate over $A$ and $\fkM$ stands for the graded maximal ideal in $\calR(I)$. Additionally, they  showed that the latter condition is equivalent to its projective scheme 
$\Proj \calR(I) =\{ P \in \Spec \calR(I) \mid P \  \text{is a graded ideal}, ~P \nsupseteq \calR(I)_+\}$
being normal, i.e., the local ring $\calR(I)_P$ is normal for every point $P \in \Proj \calR(I)$, where $\calR(I)_+ = \sum_{i >0}I^it^i$. See \cite{GM, M} for the details.

The notion of Rees algebra $\calR(I)$ can be generalized to a finitely generated $R$-module $M$; developing the theory of Rees algebras of modules is significant to further study of the Rees algebras of ideals, which is one of the motivations for this generalization. Besides, the Rees algebra of $M$ includes the notion of multi-Rees algebra, which corresponds to the case where $M$ forms a direct sum of ideals. Moreover, T. Gaffney requires this generalization of Rees algebras for applications to equisingularity theory (e.g., \cite{Gaf92, Gaf96}). Geometrically, the projective scheme of $\calR(M)$ defines the blow-up of $A$ at the module $M$ as well as the case of ideals (see \cite{Rossi, Villamayor}). Hence, it is still worth considering the notion of Rees algebras of modules for not only commutative algebra but also algebraic geometry and theory of singularities.

In this paper, to further study of the Rees algebras of modules, we investigate the question of when the Ratliff--Rush closure coincides with the integral closure for the case of modules.

We now explain our results more precisely. Let $A$ be a Noetherian ring, $M$ a finitely generated $A$-module which is contained in a free module $F$ of finite rank $r>0$. We denote the symmetric algebras of $M$ and $F$ by $\Sym_A(M)$, $\Sym_A(F)$, respectively. Let $\operatorname{Sym}(i) : \Sym_A(M) \to \Sym_A(F)$ be the homomorphism induced by the embedding $i : M \hookrightarrow F$. The Rees algebra $\calR(M)$ of $M$ is defined by
$$
\calR(M) = \Im \left[ \Sym_A(M) \overset{\operatorname{Sym}(i)}{\longrightarrow} \Sym_A(F) \right]
$$
(see \cite{SUV}). Hence, $\calR (M) = \Sym_A(M)/T$, where $T=t(\Sym_A(M))$ denotes the torsion part of $\Sym_A(M)$ as an $A$-module. Let $M^n = [\calR (M)]_n$ stand for the homogeneous component of $\calR(M)$ of degree $n$. In particular, $M = [\calR (M)]_1$ is an $A$-submodule of $\calR(M)$. We set
$$
\widetilde{\calR(M)} = \varepsilon^{-1}(\rmH^0_\fka(S/\calR(M)))
$$
which forms a graded subring of the polynomial ring $S=\Sym_A(F)$, containing the Rees algebra $\calR(M)$, where $\varepsilon : S \to S/\calR(M)$ stands for the canonical surjection and $\rmH^0_\fka(-)$ denotes the $0$-th local cohomology functor with respect to $\fka = \calR(M)_+$. 

\begin{defn}\label{1.1}
For each integer $n \ge 0$, we define $\widetilde{M^n}$ to be the homogeneous component of $\widetilde{\calR(M)}$ of degree $n$ and call it {\it the Ratliff-Rush closure of $M^n$}, i.e.,
$$
\widetilde{M^n} = \bigcup_{\ell > 0} \left[(M^n)^{\ell+1}:_{F^n} (M^n)^{\ell}\right].
$$
In particular, $\widetilde{M\hspace{0.2em}} = \bigcup_{\ell > 0} \left[M^{\ell+1}:_F M^{\ell}\right].$
\end{defn}

\noindent
In the case where $A$ is a Noetherian domain, the notion of Ratliff--Rush closure $\widetilde{M\hspace{0.2em}}$ of $M$ has already defined by J.-C. Liu (\cite{L}) to be the largest $A$-submodule $N$ of $F$, which satisfies $M \subseteq N \subseteq F$ and $M^n = N^n$ for every $n \gg 0$. We shall prove in Proposition \ref{3.15} that these definitions coincide, and hence Definition \ref{1.1} generalizes the notion given by J.-C. Liu.


If $R$ is a Noetherian local ring with maximal ideal $\m$, then $\calR(M)$ possesses a unique graded maximal ideal $\fkM=\m \calR(M) + \fka$. Then, we say that $\calR(M)$ has {\it finite local cohomology} if the $i$-th graded local cohomology module $\rmH_\fkM^i(\calR(M))$ is finitely generated for every $i \ne \dim \calR(M)$.


With this notation, the main result of this paper is stated as follows, which is a complete generalization of the results in \cite{GM, M}.

\begin{thm}\label{1.2}
Let $(A, \m)$ be a two-dimensional regular local ring with infinite residue class field, $M \ne (0)$ a finitely generated torsion-free $A$-module. Then the following conditions are equivalent.
\begin{enumerate}
\item[$(1)$] $\widetilde{M\hspace{0.2em}}=\overline{M}$.
\item[$(2)$] $\widetilde{M^n}=\overline{M^n}$ for every $n>0$.
\item[$(3)$] $\overline{M^{\ell}} = M^{\ell}$ for some $\ell>0$.
\item[$(4)$] There exists an integer $\ell>0$ such that $\overline{M^n} =M^n$ for every $n \ge \ell$.
\item[$(5)$] $\Proj\calR(M)$ is a normal scheme.
\item[$(6)$] $\calR(M)_P$ is normal for every $P \in \Spec \calR(M)\setminus \{\fkM\}$.
\end{enumerate}
When this is the case, we have the following.
\begin{enumerate}
\item[$(a)$] $\calR(M)$ has finite local cohomology and $\H_{\fkM}^p(\calR(M)) = (0)$ for every $p \ne 1, r+2$. 
\item[$(b)$] $\left[\H_{\fkM}^1(\calR(M))\right]_n \cong \overline{M^n}/M^n$ as an $A$-module for every $n \in \Bbb Z$.
\item[$(c)$] $\calR(M)$ is a Cohen--Macaulay ring if and only if $M$ is integrally closed.
\end{enumerate}
\end{thm}

As a consequence, Theorem \ref{1.2} leads us to obtain a criterion for the Buchsbaum Rees algebras $\calR(M)$ with $\widetilde{M\hspace{0.2em}} = \overline{M}$ (see Theorem \ref{5.1}). Besides, we construct numerous examples of Buchsbaum Rees algebras. Furthermore, we shall give an example of the Buchsbaum Rees algebra of indecomposable module $M$, that is, $\calR(M)$ cannot be appeared as the multi-Rees algebra. Finally let us show the sufficient condition for the fiber cone $\calF(M) = A/\m \otimes_A \calR(M)$ of $M$ to be a Buchsbaum ring.

Let us now explain how this paper is organized. In Section 2, we provide an overview of the Rees algebras of modules, including the notion of integral closures. Section 3 defines the Ratliff--Rush closure of modules and provides some preliminary results. In Section 4, we provide a proof of Theorem \ref{1.2}, and in the last section, we explore the application of our theory.


\section{Preliminaries}

In this section, we summarize some basic properties of the Rees algebras and the integral dependence for modules. Let $A$ be a commutative Noetherian ring, $M$ a finitely generated $A$-module which is contained in a free $A$-module $F$ of positive rank $r > 0$. The embedding $M \subseteq F$ induces the graded $A$-algebra homomorphism between the symmetric algebras
$$
\Sym(i) : \Sym_A(M) \longrightarrow \Sym_A(F)
$$
of $M$ and $F$. As $F$ is the free $A$-module, the symmetric algebra $S = \Sym_A(F)$ of $F$ concides with the polynomial ring $S=A[t_1, t_2, \ldots, t_r]$ over $A$, where $r=\rank_A F >0$. 
In 2003, A. Simis, B. Ulrich, and W. V. Vasconcelos defined the Rees algebra $\calR(M)$ of the module $M$ as the image of the induced homomorphism;
\begin{eqnarray*}
\calR(M) &=& \Im( \Sym (i) )\hspace{0.5em} \subseteq \hspace{0.5em} S \hspace{0.4em} = \hspace{0.4em} A[t_1, t_2, \ldots, t_r]\\
         &=& \bigoplus_{n \ge 0}M^n
\end{eqnarray*}
where $M^n$ denotes the $n$-th homogeneous component of the graded ring $\calR(M)$. Hence, if we take $M$ to be an ideal $I$ and $F=R$, then the Rees algebra of $M$ is exactly the same as the usual Rees algebra $\calR(I)$ of the ideal.

Let us recall the definition of the integral closure of modules.

\begin{defn}\label{2.1}
For every integer $n \ge 0$, we define {\it the integral closure}
$$
\overline{M^n} = \left(\overline{\calR(M)}^{S}\right)_n \subseteq S_n=F^n
$$
of $M^n$ to be the $n$-th homogeneous component of the integral closure $\overline{\calR(M)}^{S}$ of $\calR(M)$ in $S$. In other words, $\overline{M^n}$ is the integral closure of the ideal $(MS)^n$ of degree $n$, i.e., 
$$
\overline{M^n} = \left(\overline{(MS)^n}\right)_n.
$$
In particular, $\overline{M} =\left(\overline{MS}\right)_1 \subseteq F$. 
Hence $\overline{M}$ consists of the element $x \in F$ which satisfies the integral equation $x^n + c_1x^{n-1} + \cdots + c_n = 0$ in $S$, where $n >0$ and $c_i \in M^i$ for every $1 \le  i \le n$. 
\end{defn}

Let us note the following, which might be known, but we include a brief proof for the sake of completeness. We denote by $\rmQ(R)$ the total ring of fractions of a ring $R$.

\begin{lem}\label{2.3}
Suppose that $\rank_AM=r$. Then one has 
$\rmQ(\calR(M))=\rmQ(S).$
Moreover, if we assume that $A$ is a normal domain, then 
$\overline{\calR(M)}^{\rmQ(\calR(M))} = \ \overline{\calR(M)}^S$.
\end{lem}


\begin{proof}
Look at the commutative diagram 
\[
\xymatrix{
\rmQ(A) \otimes_A \Sym_A(M) &\ar[r]^{\exists1} && \rmQ(A) \otimes_A S &  \\
\Sym_A(M) \ar[u] & \ar[r]^{\Sym(i)} & & S \ar[u]
}
\]
where the vertical maps are canonical homomorphisms of localizations. The isomorphism $\rmQ(A) \otimes_A M \cong \rmQ(A) \otimes_A F$ yields that
$\rmQ(A) \otimes_A \Sym_A(M) \cong \rmQ(A) \otimes_A S$
as an $A$-module. Since $S$ is free as an $A$-module, we get the exact sequence
$$
0 \to t(\Sym_A(M)) \to \Sym_A(M) \to \calR(M) \to 0
$$
of $A$-modules, where $t(\Sym_A(M))$ denotes the torsion part of $\Sym_A(M)$ as an $A$-module. Therefore, $\rmQ(A) \otimes_A S \cong \rmQ(A) \otimes_A \Sym_A(M) \cong \rmQ(A) \otimes_A \calR(M)$, so that $\rmQ(\calR(M))=\rmQ(S)$. The last assertion follows from the fact that $S =A[t_1, t_2, \ldots, t_r]$ is a normal domain. This completes the proof.
\end{proof}

Thanks to Lemma \ref{2.3}, we have the following, which claims that, up to isomorphism, the integral closure does not depend on the choice of the embedding of $M$. Remember that an $A$-module $M$ is called {\it an ideal module}, if $M \ne (0)$ is finitely generated, torsion-free, and the double dual $M^{**}$ of $M$ is free, where $(-)^* = \Hom_A(-, A)$. Typically, finite direct sums of ideals of grade at least two and non-zero finitely generated torsion-free modules over two-dimensional regular local rings are the ideal modules. The reader is referred to \cite[Section 5]{SUV} for basic properties of ideal modules.

\begin{prop}\label{2.4}
Suppose that $A$ is a normal domain and $M$ is an ideal module. If $M$ is embedded into the finite free module $G$ of positive rank, then 
$
\overline{\calR(M)}^{S} \ \cong \ \overline{\calR(M)}^{T}
$
where $T=\Sym_A(G)$.
\end{prop}

\begin{proof}
Note that $F=M^{**}$ is a finitely generated free $A$-module and we get a canonical embedding $0 \rightarrow M \overset{\varphi}{\rightarrow} F$ of $A$-modules.  We set $\xi={\eta_G}^{-1}\circ\psi^{**}$, where $\psi:M \to G$ denotes another embedding of $M$ and $\eta_G : G \to G^{**}$ stands for the biduality map. We then have the commutative diagram
\[
\xymatrix{
0 \ar[r] & M \ar[r]^{\varphi} \ar[d]^{{\rm id}_M} & F \ar[d]^{\xi}  \\
0  \ar[r] & M   \ar[r]^{\psi}  & G 
}
\]
of $A$-modules. Passing to the localization $\rmQ(A) \otimes_A (-)$, we obtain the injectivity of $\psi^{**}$, so is $\xi$. Hence we get the homomorphism
$$
\Sym(\xi) : S= \Sym_A(F) \longrightarrow T=\Sym_A(G)
$$
which is induced by $\xi : F \to G$. The splitting exact sequence
$$
0 \longrightarrow \rmQ(A) \otimes_A F \longrightarrow \rmQ(A) \otimes_A G \longrightarrow \rmQ(A) \otimes_A X \longrightarrow 0
$$
implies that $\rmQ(A) \otimes S(\xi)$ is a split monomorphism, where $X = \Coker~ \xi$. Therefore $\Sym(\xi)$ is injective. Passing to the map $\Sym(\xi)$, let us consider $\calR(M) \subseteq S \subseteq T$. To prove the equality of the integral closures of $\calR(M)$, it is enough to show that $S$ is integrally closed in $T$. We take $x \in T$, satisfying an integral equation
$$
x^n + c_1 x^{n-1} + \cdots + c_n = 0
$$
where $n>0$ and $c_i \in S$ for every $1 \le i \le n$. Then $1\otimes x\in \rmQ(A)\otimes_A T$, $1\otimes c_i\in \rmQ(A)\otimes_A S$, and
$$\hspace{-1em}
(*) \quad \quad (1 \otimes x)^n + (1 \otimes c_1) (1 \otimes x)^{n-1} + \cdots + (1 \otimes c_n) = 0 \quad \text{in} \quad \rmQ(A)\otimes_A T.
$$
Since $F \subseteq G$, note that $\rmQ(A)\otimes_A T$ is the polynomial ring over $\rmQ(A)\otimes_A S$ with $q \ge 0$ variables. To see the degree of the integral equation $(*)$, $1 \otimes x$ is a constant in $\rmQ(A)\otimes_A T$, so that $1 \otimes x \in \rmQ(A)\otimes_A S$. Because $S$ is a normal domain, $S$ is integrally closed in $T$, as desired.
\end{proof}


\section{Ratliff--Rush closure of modules}

The aim of this section is to introduce the notion of Ratliff--Rush closure of modules and to give some basic properties. Firstly let us fix the notation. Let $A$ be a Noetherian ring, $M$ a finitely generated $A$-module which is contained in a finite free module $F$ of rank $r>0$.
Let us denote by $\fka = \calR(M)_+ = \bigoplus_{n>0}M^n$ the positive part of $\calR(M)$. We define
$$
\widetilde{\calR(M)}^{S} = \varepsilon^{-1}\left(\rmH^0_{\fka}(S/\calR(M))\right) \subseteq S
$$
which forms a graded subring of $S$, containing $\calR(M)$, where $S=\Sym_A(F)$ is the symmetric algebra of $F$, $\rmH^0_{\fka}(-)$ denotes the $0$-th local cohomology functor with respect to the ideal $\fka$, and $\varepsilon : S \rightarrow S/\calR(M)$ stands for the canonical surjection. 


\begin{defn}\label{3.1}
For every integer $n \ge 0$, we define the {\it the Ratliff--Rush closure} 
$$
\widetilde{M^n\hspace{0.1em}} = \left(\widetilde{\calR(M)}^{S}\right)_n \subseteq S_n=F^n
$$
of $M^n$ to be the $n$-th homogeneous component of $\widetilde{\calR(M)}^{S}$.
\end{defn}

In the present paper, we adapt the above definition of Ratliff--Rush closures. However, the notion has already defined by J.--C. Liu in 1998 in the case where $A$ is a Noetherian domain ({\cite[Definition 2.2]{L}}). She defined the Ratliff--Rush closure $\widetilde{M\hspace{0.2em}}$ of $M$ to be the largest $A$-submodule $N$ of $F$ which satisfies the following two conditions;
\begin{itemize}
\item[$(1)$] $M \subseteq N \subseteq F$, 
\item[$(2)$] $M^n = N^n$ for every $n \gg 0$.
\end{itemize}
Note that these definitions coincide, when $A$ is a Noetherian domain (see Proposition \ref{3.15}).

The following ensures that Definition \ref{3.1} is a natural generalization of the ordinary Ratliff--Rush closure of ideals. The proof immediately comes from the definition.

\begin{prop}\label{3.4}
For every non-negative integer $n\ge 0$, we have
$$
\widetilde{M^n\hspace{0.1em}} = \bigcup_{\ell>0}\left[(M^n)^{\ell + 1}:_{F^n}(M^n)^{\ell}\right] = \left(\widetilde{(MS)^n}\right)_n.
$$
In particular
$$
\widetilde{M \hspace{0.2em}} = \bigcup_{\ell>0}\left[M^{\ell + 1}:_{F}M^{\ell}\right] = \left(\widetilde{MS\hspace{0.1em}}\right)_1.
$$
\end{prop}


By Proposition \ref{3.4}, we reduce to the case of ideals, and therefore we get the following.
Notice that $M$ is faithful as an $A$-module if and only if the ideal $MS$ of $S$ contains a non-zero divisor on $S$.

\begin{cor}\label{3.5}
Suppose that $M$ is a faithful $A$-module. Then
$\widetilde{M^n\hspace{0.2em}} \subseteq \overline{M^n} \subseteq F^n$
for every non-negative integer $n \ge 0$. 
Hence, $\calR(M) \subseteq \widetilde{\calR(M)}^{S} \subseteq \overline{\calR(M)}^S \subseteq S$.
\end{cor}

\begin{proof}
Since $MS$ contains a non-zero divisor on $S$, $\widetilde{(MS)^n}\subseteq \overline{(MS)^n}$ by \cite[Section 1]{RR}. Hence the result comes from Proposition \ref{3.4}. 
\end{proof}

Similarly for the case of integral closures, we have the following.

\begin{prop}\label{3.6}
Suppose that $A$ is a normal domain and $M$ is an ideal module. If $M$ is embedded into the finite free module $G$ of positive rank, then 
$
\widetilde{\calR(M)}^{S} \ \cong \ \widetilde{\calR(M)}^{T}
$
where $T=\Sym_A(G)$.
\end{prop}

\begin{proof}
Let us maintain the notation as in the proof of Proposition \ref{2.4}. By the proof of Proposition \ref{2.4}, we may assume that
$\calR(M) \subseteq S \subseteq T$. Then, thanks to Proposition \ref{2.4} and Corollary \ref{3.5}, we get
$
\widetilde{\calR(M)}^{T} \subseteq \overline{\calR(M)}^{T} = \overline{\calR(M)}^{S} \subseteq S
$
which implies that $\widetilde{\calR(M)}^{S} \supseteq \widetilde{\calR(M)}^{T}$ as desired.
\end{proof}

Therefore, up to isomorphism, the Ratliff--Rush closure does not depend on the choice of the embedding of $M$.

\medskip

Let us now explore some examples. To do this, we first recall the notion of parameter module. Notice that, if $M$ is an ideal of $A$, then the parameter module in $A$ is exactly the same as the usual parameter ideal. We denote by $\mu_A(-)$ (resp. $\ell_A(-)$) the number of elements in a minimal system of generators (resp. the length as an $A$-module).


\begin{defn}[\cite{BR, HH}]\label{3.7}
Suppose that $(A, \m)$ is a Noetherian local ring with $d=\dim A$. Then $M$ is called {\it a parameter module in $F$}, if $\ell_A(F/M) < \infty$, $M \subseteq \m F$, and $\mu_A(M) = d + r -1$.
\end{defn}

\begin{rem}
Let $(A, \m)$ be a Noetherian local ring with $d = \dim A > 0$, $N$ a parameter module in $F$. If $N$ is integrally closed, then $r \le d-1$. In particular, if $d=2$, then $F \cong A$.
\end{rem}

\begin{proof}
Note that $\mu_A(N) = d+r-1 \ge r$. Thanks to \cite[Propositions 2.5, 4.3]{K}, we have $\mu_A(N) = \ord_A(\operatorname{Fitt}_A(F/N)) + r \ge 2r$, because $N \subseteq \m F$. Hence $r \le d-1$, as desired.
\end{proof}

We say that the module $M$ is {\it Ratliff--Rush closed}, if $\widetilde{M \hspace{0.2em}} = M$.

\begin{prop}\label{3.8}
Suppose that $(A, \m)$ is a Cohen--Macaulay local ring with $d = \dim A>0$. If $M$ is a parameter module in $F$, then 
$\widetilde{\calR(M)}^{S} = \calR(M)$. In particular, $M$ is Ratliff--Rush closed.
\end{prop}

\begin{proof}
First, we consider the case where $r = 1$. Then $M$ forms the parameter ideal $Q$ of $A$. Notice that the associated graded ring $G=\gr_Q(A) \cong (A/Q)[X_1, X_2, \ldots, X_d]$ of $Q$ is a Cohen--Macaulay ring with $\dim G = d >0$. Hence $Q^n$ is Ratliff--Rush closed for every $n \ge 0$. We may assume that $r \ge 2$. Remember that $r = \height_{\calR(M)} \fka$ (see \cite[Proposition 2.2]{SUV}). Since $\calR(M)$ is a Cohen--Macaulay ring, we have $\rmH^i_{\fka}(\calR(M)) = (0)$ for $i=0, 1$, whence $\rmH^0_{\fka}(S/\calR(M)) = (0)$. Consequently, $\widetilde{\calR(M)}^{S} = \calR(M)$.
\end{proof}

We note some examples.

\begin{ex}
Let $A=k[[X, Y]]$ be the formal power series ring over a field $k$. We set
$$
M = \left< 
\begin{pmatrix}
X \\
0
\end{pmatrix}, 
\begin{pmatrix}
Y \\
X
\end{pmatrix}, 
\begin{pmatrix}
0 \\
Y
\end{pmatrix}
\right> \subseteq F = A \oplus A.
$$
Then $M$ is a parameter module in $F$ and hence $M$ is Ratliff--Rush closed.   
\end{ex}

There is an example of the parameter module which is Ratliff--Rush closed, even if $A$ is not a Cohen--Macaulay ring.

\begin{ex}
Let $R=k[[X, Y, Z, W]]$ be the formal power series ring over a field $k$. We set $A = R/(X, Y) \cap (Z, W)$ which is a two-dimensional Buchsbaum local ring with $\depth A=1$. Then the parameter ideal $Q = (X-Z, Y-W)A$ is Ratliff--Rush closed. 
\end{ex}

The following property is useful for finding the Ratliff--Rush closure. 
An $A$-submodule $L$ of $M$ is called {\it a reduction of $M$}, if $M^{r+1} = L M^r$ for some integer $r \ge 0$, which is equivalent to saying that $M \subseteq \overline{L}$. 

\begin{prop}\label{3.12}
Let $L = Ax_1 + Ax_2 + \cdots + Ax_{\ell} \subseteq M$ be an $A$-submodule of $M$ such that $L$ is a reduction of $M$. Then
$$
\widetilde{M \hspace{0.2em}} = \bigcup_{n>0}\left[M^{n+1}:_F (A{x_1}^n + A{x_2}^n + \cdots + A{x_\ell}^n)\right].
$$
\end{prop}

\begin{proof}
We set $N=\bigcup_{n>0}\left[M^{n+1}:_F (A{x_1}^n + A{x_2}^n + \cdots + A{x_\ell}^n)\right]$. Let $x \in N$ and choose an integer $n>0$ such that $x\cdot{x_i}^n \in M^{n+1}$ for each $1 \le i \le \ell$. We than have $x L^{m} \subseteq M^{m+1}$
for every $m \ge (n-1) \ell + 1$. Since $L$ is a reduction of $M$, we take an integer $s \ge 0$ such that $M^{s+1} = L \cdot M^s$. Therefore
$$
x M^{m+s} = x (L^mM^s) = (x L^m)M^s \subseteq M^{m+1}M^s = M^{m+s+1}
$$
whence $x \in \widetilde{M \hspace{0.2em}}$ by Proposition \ref{3.4}.
\end{proof}

As a consequence of Proposition \ref{3.12}, we immediately get the following.

\begin{cor}
If $L$ is a reduction of $M$, then
$\widetilde{L\hspace{0.2em}} \subseteq \widetilde{M\hspace{0.2em}}.$
\end{cor}

\begin{rem}[{\cite[$(1.1)$]{HJLS}}]
In general, the embedding $L \subseteq M$ does not imply $\widetilde{L\hspace{0.2em}} \subseteq \ \widetilde{M\hspace{0.2em}}$, even if $\rank_AF=1$. For example, let $k[[t]]$ be the formal power series ring over a field $k$. We set $A = k[[t^3, t^4]]$, $I=(t^8)$, and $J=(t^{11}, t^{12})$.
Then $J \subseteq I$, but $\widetilde{J} \nsubseteq \widetilde{I}$.
\end{rem}

The following is the key in our argument. Assertion $(2)$ of Proposition \ref{3.15} shows that the Ratliff--Rush closure in the sense of Definition \ref{3.1} is equivalent to \cite[Definition 2.2]{L}, when $A$ is a Noetherian domain.

\begin{prop}\label{3.15}
Suppose that $M$ is a faithful $A$-module. Then the following assertions hold.
\begin{enumerate}
\item[$(1)$] $\widetilde{M^n \hspace{0.2em}} = (\widetilde{M})^n = M^n$ for every $ n \gg 0$.
\item[$(2)$] Let $N$ be an $A$-submodule of $F$ such that $M \subseteq N$. Then the following conditions are equivalent. 
\begin{enumerate}
\item[$(\rm i)$] $N \subseteq \widetilde{M \hspace{0.2em}}$.
\item[$(\rm ii)$] $M^{\ell} = N^{\ell}$ for some $\ell >0$.
\item[$(\rm iii)$] $M^n = N^n$ for every $n \gg 0$.
\item[$(\rm iv)$] $\widetilde{M\hspace{0.2em}} = \widetilde{N\hspace{0.3em}}$.
\end{enumerate}
\item[$(3)$] $\widetilde{M \hspace{0.2em}}$ is Ratliff--Rush closed, i.e., $\widetilde{\widetilde{M\hspace{0.2em}}} = \widetilde{M \hspace{0.2em}}$.
\end{enumerate}
\end{prop}

\begin{proof}
The assertion $(3)$ immediately comes from $(2)$. To prove the assertions $(1)$ and $(2)$, by Proposition \ref{3.4}, we are able to reduce to the case of ideals. Then the assertions follow from \cite[Theorem 2.1, Corollary 2.2]{RR}.
\end{proof}


The following is essentially due to Y. Shimoda.

\begin{rem}\label{shimoda}
Let $N$ be an $A$-submodule of $F$ such that $M \subseteq N$. If $M^{\ell} = N^{\ell}$ for some $\ell > 0$, then $M^n = N^n$ for every $n \ge \ell$.
\end{rem}

\begin{proof}
By the assumption, we obtain the case where $n=\ell$. Suppose that $n>\ell$ and the assertion holds for $n-1$. Then, by the induction hypothesis, we get
$
M^{n-1} = N^{n-1} \subseteq MN^{n-2} \subseteq MM^{n-2} = M^{n-1}$, so that $N^{n-1} = M N^{n-2}$. Therefore
$
N^n = N^{n-1} N = (M N^{n-2})N = M N^{n-1} = M M^{n-1} = M^n
$
which completes the proof.
\end{proof}

Let us note the following.

\begin{lem}\label{3.16}
Suppose that $(A, \m)$ is a Noetherian local ring. If $\overline{M} = F$, then $M = F$. 
In particular, if $M$ is a faithful $A$-module and $M \ne F$, then $\widetilde{M \hspace{0.2em}} \ne F$.
\end{lem}

\begin{proof}
Suppose the contrary and choose a counterexample $M$ so that $r=\rank_A F>0$ is as small as possible. Then the minimality shows $M \subseteq \m F$. Therefore $F=\overline{M} \subseteq \overline{\m F} = \overline{\m}F = \m F$, whence we get $F = \m F$, which makes a contradiction. Let us make sure of the last assertion. Suppose that $M \ne F$ and $M$ is faithful. Then $\overline{M} \ne F$ and we get $\widetilde{M \hspace{0.2em}} \ne F$ by Corollary \ref{3.5}. 
\end{proof}

In what follows, we focus on the Buchsbaum--Rim coefficients for the module. Suppose that $A$ is a Noetherian local ring with maximal ideal $\m$, $M \ne (0)$ and $0<\ell_A(F/M) < \infty$.
With this notation, in 1964, Buchsbaum and Rim showed that there exists an integer $\rmbr_i(M) \in \Bbb Z~(0 \le i \le d+r-1)$ such that $\ell_A(F^{n+1}/M^{n+1})$ can be expressed as the polynomial of the form;
$$
\ell_A(F^{n+1}/M^{n+1}) = \sum_{i=0}^{d+r-1}(-1)^i\cdot \rmbr_i(M)\cdot\binom{n+d+r-i-1}{d+r-2}
$$
for every $n \gg 0$ (see \cite{BR}). 
The integer $\rmbr_i(M)$ is called {\it the $i$-th Buchsbaum--Rim coefficient of $M$}. 
We set
$$
\calS =\{ N \subseteq F \mid M \subseteq N \subsetneq F, \ \rmbr_i(M) = \rmbr_i(N) \ \text{for every} \ 0 \le i \le d+r-1 \}.
$$

We then have the following, which shows that $\widetilde{M \hspace{0.2em}}$ is the largest $A$-submodule $N$ of $F$ having the same Buchsbaum--Rim polynomial of $M$.

\begin{prop}
Suppose that $M$ is a faithful $A$-module. Then
$\widetilde{M \hspace{0.2em}} \in \calS$ and $N \subseteq \widetilde{M \hspace{0.2em}}$ for every $N \in \calS$.
\end{prop}

\begin{proof}
Note that $\widetilde{M} \ne F$ by Lemma \ref{3.16}.
Choose an integer $\ell > 0$ such that $\widetilde{M^n} = M^n$ for every $n \ge \ell$. Then we have $M^n \subseteq (\widetilde{M})^n \subseteq \widetilde{M^n} = M^n$ and hence $(\widetilde{M})^n = M^n$. Thus $\ell_A(F^{n+1}/M^{n+1}) = \ell_A(F^{n+1}/{(\widetilde{M})}^{n+1})$ for every $n \gg 0$, as desired.
\end{proof}




\section{Proof of Theorem \ref{1.2}}

The purpose of this section is to prove Theorem \ref{1.2}. First of all, we fix our notation and assumptions on which all the results in this section are based.

\begin{setting}\label{4.1}
Let $(A, \m)$ be a two-dimensional regular local ring with infinite residue class field, $M \ne (0)$ a finitely generated torsion-free $A$-module. We set $F= M^{**}$ a finitely generated free $A$-module with $r=\rank_A F>0$, where $(-)^{*}=\Hom_A(-, A)$. 
\end{setting}

\noindent
Passing to the biduality map $M \rightarrow F=M^{**}$, we have $M \subseteq F$ and $\ell_A(F/M) < \infty$ (see e.g., \cite[Proposition 2.1]{K}). Let us consider the projective scheme 
$
\Proj \calR(M) =\{ P \in \Spec \calR(M) \mid P  \ \text{is a graded ideal}, ~P \nsupseteq \fka\}
$
of the Rees algebra $\calR(M)$, where $\fka = \calR(M)_+$. It is known by \cite[Proposition 2.2]{SUV} that $\dim \calR(M) = r+2$. Moreover, in 1995, V. Kodiyalam proved an analogue of the famous result of O. Zariski, i.e., the products of integral closures are the integral closure of the products of modules, namely the equality $\overline{M^n} = (\overline{M})^n$ holds for every $n \ge 0$ (\cite[Theorem 5.2]{K}). Therefore, the integral closure of $\calR(M)$ in its total ring of fractions is given by $\calR(\overline{M})$ which is a module finite extension over $\calR(M)$.

\medskip

We are now in a position to prove Theorem \ref{1.2}.

\begin{proof}[Proof of Theorem \ref{1.2}]
The implications $(2) \Rightarrow (1)$, $(4) \Rightarrow (3)$,  $(6) \Rightarrow (5)$ are obvious.

$(1) \Rightarrow (4)$~~By Proposition \ref{3.15}, $M^n = (\widetilde{M\hspace{0.2em}})^n$ for every $n \gg 0$. Therefore we get 
$M^n = (\widetilde{M\hspace{0.2em}})^n = (\overline{M})^n = \overline{M^n}$, as wanted.

$(3) \Rightarrow (1)$~~Suppose that $M^n = \overline{M^n} = (\overline{M})^n$ for some $n>0$. Then, by Remark \ref{shimoda}, we have $(\overline{M})^{n+1} = M^{n+1}$. Thus
$\overline{M} \subseteq M^{n+1}:_F (\overline{M})^n = M^{n+1}:_FM^n \subseteq \widetilde{M \hspace{0.2em}}\subseteq \overline{M}$ so that $\widetilde{M\hspace{0.2em}} = \overline{M}$.

$(1) \Rightarrow (2)$~~By our assumption, we have $(\widetilde{M\hspace{0.2em}})^n = (\overline{M})^n$ for every $n > 0$. Then 
$\overline{M^n} = (\overline{M})^n = (\widetilde{M\hspace{0.2em}})^n \subseteq \widetilde{M^n\hspace{0.2em}} \subseteq \overline{M^n}$
as desired.

$(4) \Rightarrow (6)$~~Suppose $M^n = \overline{M^n}$ for every $n\gg 0$. Let $C = \calR(\overline{M})/\calR(M)$. We then have $C_n=(0)$ for $n\gg 0$, so that $C$ is finitely graded. Therefore, since $\ell_A(F/M)<\infty$, we obtain
$$
\fka^m \cdot C =(0), \ \ \m^m\cdot C =(0)
$$ for some integer $m > 0$, where $\fka = \calR(M)_+$. Thus $\fkM \subseteq \sqrt{(0):C}$ and hence $\Supp_{\calR(M)} C \subseteq \{\fkM\}$. Consequently, $C_P=(0)$ for every $P \in \Spec \calR(M) \setminus \{\fkM\}$, whence $\calR(M)_P = \calR(\overline{M})_P$ is normal.

$(5) \Rightarrow (4)$~~Let $C = \calR(\overline{M})/\calR(M)$. For each $Q \in \Supp_{\calR(M)}C$, there exists $P \in \Ass_{\calR(M)}C$ satisfying $P \subseteq Q$. Note that $P$ is a graded prime ideal of $\calR(M)$.
We claim that $P \supseteq \fka$. Indeed, if we assume that contrary, that is, $P \not\supseteq \fka$. Then $P \in \Proj (\calR(M))$, so that $\calR(M)_P$ is normal by our hypothesis. Hence $\calR(M)_P = \calR(\overline{M})_P$, which yields $C_P =(0)$. This makes a contradiction, because $P \in \Ass_{\calR(M)}C$. Thus $P \supseteq \fka$ and hence $Q \supseteq \fka$. Therefore $\fka \subseteq \sqrt{(0):C}$, so that $\fka^{\ell} \cdot C =(0) $
for every $\ell \gg 0$. Consequently, $C$ is finitely graded, because $C$ is finitely generated as an $\calR(M)$-module. We finally get $M^n = \overline{M^n}$ for every $n\gg 0$. This completes the proof of the equivalent conditions.

Let us make sure of the last assertions. Since $A$ has an infinite residue class field, we choose a parameter module $L$ in $F$ such that $L$ is a reduction of $\overline{M}$. Thanks to \cite[Proposition 2.2]{KK}, we have $(\overline{M})^2 = L\cdot \overline{M}$, so that $\calR(\overline{M})$ is a Cohen--Macaulay ring. Therefore, $\rmH^1_\fkM(\calR(M)) \cong \calR(\overline{M})/\calR(M)$ and $\rmH^i_\fkM(\calR(M)) = (0)$ for $i \ne 1, r+2$. Hence $\calR(M)$ has finite local cohomology, and $\calR(M)$ is a Cohen--Macaulay ring if and only if $\rmH^1_\fkM(\calR(M))=(0)$. The latter condition is equivalent to saying that $(\overline{M})^n = M^n$ for every $n>0$; in other words, $M$ is integrally closed.
\end{proof}

\begin{rem}
By \cite[Proposition 3.2]{KK} or \cite[Proposition 4.4 (a)]{SUV}, it is proved that $\calR(M)$ is a Cohen--Macaulay ring, provided $M$ is integrally closed. Besides this, if $M$ is integrally closed, then $\calR(M)$ is an almost Gorenstein graded ring (see \cite[Corollary 2.7]{GMTY}).
\end{rem}

As a corollary of Theorem \ref{1.2}, we completely determine the Buchsbaum--Rim coefficients, when the Ratliff--Rush closure coincides the integral closure.

\begin{cor}
Suppose that $M \ne F$ and $\widetilde{M\hspace{0.2em}} = \overline{M}$. Then $\rmbr_1(M) = \rmbr_0(M) - \ell_A(F/\overline{M})$, $\rmbr_i(M) = 0$ for every $2 \le i \le r+1$, and
$$
\ell_A(F^{n+1}/(\overline{M})^{n+1}) = \rmbr_0(M)\cdot\binom{n+r+1}{r+1} - \rmbr_1(M)\cdot\binom{n+r}{r}
$$
for every $n \ge 0$.
\end{cor}

\begin{proof}
We have $\rmbr_i(M) =\rmbr_i(\widetilde{M}) = \rmbr_i(\overline{M})$ for every $0 \le i \le r+1$. Since $\overline{M}$ has the reduction number at most one, by \cite[Corollary 4.2]{KK}, we get 
$$
\ell_A(F^{n+1}/(\overline{M})^{n+1}) = \rmbr_0(\overline{M})\cdot\binom{n+r+1}{r+1} - \ell_A(\overline{M}/L)\cdot\binom{n+r}{r}
$$
for every $n \ge 0$, where $L$ is a parameter module in $F$ such that $L$ is a reduction of $\overline{M}$. Moreover, it is known by \cite[Theorem 3.1]{BUV} that $\rmbr_0(\overline{M}) = \rmbr_0(L) = \ell_A(F/L)$. Hence 
$$
\rmbr_1(\overline{M}) = \ell_A(\overline{M}/L) = \ell_A(F/L) - \ell_A(F/\overline{M}) = \rmbr_0(\overline{M}) - \ell_A(F/\overline{M})
$$
which completes the proof.
\end{proof}




\section{Applications}

In this section we explore the application of Theorem \ref{1.2}. Let us maintain the notation as in Setting $\ref{4.1}$. 

\begin{thm}\label{5.1}
The following conditions are equivalent.
\begin{enumerate}
\item[$(1)$] $\calR(M)$ is a Buchsbaum ring and $\widetilde{M\hspace{0.2em}} = \overline{M}$.
\item[$(2)$] $\calR(M)$ is a Buchsbaum ring and $\Proj \calR(M)$ is normal.
\item[$(3)$] $\m \overline{M} \subseteq M$ and $M\cdot \overline{M} = M^2$.
\end{enumerate}
When this is the case, one has
$
\rmH^1_\fkM(\calR(M)) = \left[\rmH^1_\fkM(\calR(M))\right]_1 \cong \overline{M}/M
$
and $\overline{M^n}=M^n$ for every $n \ge 2$.
\end{thm}

\begin{proof}
The equivalence between $(1)$ with $(2)$ follows from Theorem \ref{1.2}.

$(1) \Rightarrow (3)$~~Since $\widetilde{M\hspace{0.2em}} = \overline{M}$, by Theorem \ref{1.2}, we have an isomorphism 
$$
\rmH^1_\fkM(\calR(M)) \cong \calR(\overline{M})/\calR(M)
$$
of graded $\calR(M)$-modules. As $\calR(M)$ is Buchsbaum, $\fkM\cdot\rmH^1_\fkM(\calR(M))=(0)$. Hence $\fkM \cdot \calR(\overline{M}) \subseteq \calR(M)$, so that $\m \overline{M} \subseteq M$ and $M\cdot \overline{M} = M^2$.

$(3) \Rightarrow (1)$~~Since $M\cdot \overline{M} = M^2$, we get 
$\overline{M} \subseteq M^2:_F M \subseteq \widetilde{M \hspace{0.2em}} \subseteq \overline{M}$. Then $\widetilde{M \hspace{0.2em}}= \overline{M}$. By induction on $n \ge 0$, we have that $\m \cdot \overline{M^n} \subseteq M^n$ and $M\cdot \overline{M^n} = M^{n+1}$ for every $n \ge 0$. Hence, 
$\fkM \cdot \rmH^1_\fkM(\calR(M)) = (0)$. Remember that, because $\widetilde{M\hspace{0.2em}} = \overline{M}$, we have $\rmH_\fkM^p(\calR(M)) =(0)$ for every $p\ne 1, r+2$. By \cite[Corollary 1.1]{SV}, we conclude that $\calR(M)$ is a Buchsbaum ring.
\end{proof}

Let us note concrete examples in order to illustrate Theorem \ref{5.1}.

\begin{ex}
Let $A = k[[X, Y]]$ be the formal power series ring over an infinite field $k$. We set $I = (X^4, X^3Y^2, XY^6, Y^8)$ and
$M = I \oplus I \subseteq F = A\oplus A$. Then $\widetilde{M\hspace{0.2em}} = \overline{M}$, so that $\calR(M)$ has finite local cohomology, but not Buchsbaum.
\end{ex}

\begin{ex}\label{5.3}
Let $A = k[[X, Y]]$ be the formal power series ring over an infinite field $k$. We set $I_1 = (X^6, X^5Y^2, X^4Y^3, X^3Y^4, XY^7, Y^8)$, $I_2= (X^5, X^4Y^2, X^3Y^3, XY^6, Y^7)$, and 
$M = I_1 \oplus I_2 \subseteq F = A\oplus A$. 
Then $\widetilde{M\hspace{0.2em}} = \overline{M}$ and $\calR(M)$ is a Buchsbaum ring.
\end{ex}

The following ensures that there exist numerous examples of Buchsbaum Rees algebras.

\begin{cor}
Suppose that $\calR(M)$ is a Buchsbaum ring and $\widetilde{M\hspace{0.2em}} = \overline{M}$. Then, for every integrally closed $\m$-primary ideal $I$, $\calR(IM)$ is a Buchsbaum ring and $\widetilde{IM\hspace{0.2em}} = \overline{IM}$.
In particular, $\calR(\m^{\ell}M)$ is a Buchsbaum ring for every $\ell \ge 0$.
\end{cor}

\begin{proof}
Since $I$ is an $\m$-primary ideal in $A$, we get $\ell_A(F/IM)<\infty$. Then the assertion follows from $\m (\overline{IM}) = I (\m \overline{M}) \subseteq IM$ and $(IM)(\overline{IM}) = I^2 (M \overline{M}) = (IM)^2$.
\end{proof}

For the ideals $I_i~(i=1, 2)$ of $A$ as in Example \ref{5.3}, the Rees algebra $\calR(I_i)$ is Buchsbaum and $\widetilde{I_i} = \overline{I_i}$ (see \cite{GM, M}). Let us consider the relation between the Buchsbaum properties of $\calR(M_1 \oplus M_2)$ and $\calR(M_i)~ (i=1, 2)$. 

\begin{cor}
Let $M_1, M_2 \ne (0)$ be finitely generated torsion-free $A$-modules. We set $F_1 = (M_1)^{**}$, $F_2 = (M_2)^{**}$, and 
$M = M_1 \oplus M_2 \subseteq F = F_1 \oplus F_2$. Then the following conditions are equivalent.
\begin{enumerate}
\item[$(1)$] $\calR(M)$ is a Buchsbaum ring and $\widetilde{M\hspace{0.2em}} = \overline{M}$.
\item[$(2)$] $\calR(M_i)$ is a Buchsbaum ring, $\widetilde{M_i\hspace{0.2em}} = \overline{M_i}~(i=1, 2)$, and 
$M_1\cdot\overline{M_2} = \overline{M_1}\cdot M_2 = M_1\cdot M_2$.
\end{enumerate}
\end{cor}

\begin{proof}
Since $\overline{M} = \overline{M_1}\oplus\overline{M_1}$, we have the condition $\m \overline{M} \subseteq M$ if and only if $\m \overline{M_i} \subseteq M_i$ for $i=1, 2$. Moreover, by comparing the following equalities
\begin{eqnarray*}
M\overline{M} &=& M_1 \overline{M_1} \oplus (M_1 \overline{M_2} + M_2 \overline{M_1}) \oplus  M_2 \overline{M_2} \\
M^2 &=& {M_1}^2 \oplus M_1M_2 \oplus {M_2}^2
\end{eqnarray*}
we conclude that $M\cdot \overline{M}=M^2$ is equivalent to the condition of $M_i\cdot\overline{M_i}=M^2_i$ for $i=1, 2$ and $M_1\cdot\overline{M_2} = \overline{M_1}\cdot M_2 = M_1\cdot M_2$.
\end{proof}

We now summarize some consequences.

\begin{cor}
Suppose that $\calR(M)$ is a Buchsbaum ring and $\widetilde{M\hspace{0.2em}} = \overline{M}$. Then $\calR(N)$ is a Buchsbaum ring and $\widetilde{N\hspace{0.2em}} = \overline{N}$ 
for all direct summand $N$ of $M$. 
\end{cor}

\begin{cor}
Suppose that $\calR(M)$ is a Buchsbaum ring and $\widetilde{M\hspace{0.2em}} = \overline{M}$. Then $\calR(M^{\oplus \ell})$ is a Buchsbaum ring and $\widetilde{M^{\oplus \ell}\hspace{0.2em}} = \overline{M^{\oplus \ell}}$ for every $\ell > 0$.
\end{cor}

\begin{cor}
Let $L$ be an $A$-submodule of $M$ such that $M^2 = LM$. Suppose that $\calR(M), \calR(L)$ are Buchsbaum rings and $\widetilde{L\hspace{0.2em}} = \overline{L}$. Then $\calR(M \oplus L)$ is a Buchsbaum ring and $\widetilde{M \oplus L\hspace{0.2em}} = \overline{M \oplus L}$.
\end{cor}

\begin{proof}
Since $L$ is a reduction of $M$, we get $\overline{L} = \widetilde{L\hspace{0.2em}} \subseteq \widetilde{M\hspace{0.2em}} \subseteq \overline{M} =\overline{L}$, so that $\widetilde{M\hspace{0.2em}} = \overline{M}$. Then $ M\overline{M}=M^2$ and hence $M \overline{L} = M \overline{M} = M^2 = LM$, $L\overline{M} = L \overline{L} = L^2 = LM$. Hence $\calR(M \oplus L)$ is a Buchsbaum ring and $\widetilde{M \oplus L\hspace{0.2em}} = \overline{M \oplus L}$.
\end{proof}

\begin{rem}[{\cite[Example 4.3]{L}}]
In general, $\widetilde{I \oplus J} \ne \widetilde{I} \oplus \widetilde{J}$. For example, let $A = k[[X, Y]]$ be the formal power series ring over a field $k$, $I = (X^4, X^3Y, XY^3, Y^4)$, and $J=(X^5, X^2Y^2, Y^5)$. Then $\widetilde{I} = \m^4$, $\widetilde{J} = J$, and $\widetilde{I \oplus J} \subsetneq \widetilde{I} \oplus \widetilde{J}$.
\end{rem}


\medskip

Let us note the example of the Buchsbaum Rees algebra which does not appear as the multi-Rees algebra. To do this, we need some auxiliaries. For each ideal $I$ of $A$, let 
$
\ord_A(I) = \max\{n \in \Bbb Z \mid I \subseteq \m^n\}
$
and call it {\it the order of $I$}. Recall that an ideal $I$ is {\it simple}, if $I \ne JK$ for any proper ideals $J, K$. We denote by $\Fitt_i(N)$ {\it the $i$-th Fitting invariant} of an $A$-module $N$.

\medskip

With this notation, we have the following.

\begin{lem}\label{5.11}
Suppose that $\Fitt_0(F/M)$ is integrally closed and write $\Fitt_0(F/M) = I_1{\cdot} I_2 \cdots I_{\ell}$, where $\ell > 0$ and $I_i$ is a simple $\m$-primary integrally closed ideal of $A$ for every $1 \le i \le \ell$. If $r=\rank_AF = 2$ and
$
\ord_A(\Fitt_1(F/M)) > \min\{\ord_A(I_i)\mid 1 \le i \le \ell\},
$
then $M$ is indecomposable. 
\end{lem}

\begin{proof}
Suppose the contrary. We write $M \cong X \oplus Y$ for some $A$-modules $X \ne (0)$ and $Y \ne (0)$. Since $M$ is a torsion-free $A$-module with $\rank_AM=2$, we get $X, Y$ are torsion-free $A$-modules of rank one. Therefore, since $A$ is a two-dimensional regular local ring, we may choose $\m$-primary ideals $I$ and $J$ such that $X=I$, $Y=J$. Hence $C \cong A/I \oplus A/J$ as an $A$-module. To see the minimal free resolutions of $C$, $A/I$, and $A/J$, we choose the invertible matrices $P$, $Q$ satisfying
$$
P \cdot {\Bbb M}\cdot Q = 
\begin{pmatrix}
f_1 & f_2 & \cdots & f_\ell & 0 & 0 & \cdots & 0 \\
0 & 0 & \cdots & 0 & g_1 & g_2 & \cdots & g_m
\end{pmatrix}
$$
where $I = (f_1, f_2, \ldots, f_\ell)$, $J=(g_1, g_2, \cdots, g_m)$, $\ell=\mu_A(I) > 1$, $m = \mu_A(J) >1$, and 
$$
A^{\oplus(\ell + m)} \overset{\Bbb M}{\longrightarrow} A^{\oplus 2} \longrightarrow F/M \longrightarrow 0
$$ 
denotes the presentation of $F/M$ as an $A$-module. Then, since $P$, $Q$ are invertible, we get
$$
\Fitt_1(F/M) = \rmI_1(\Bbb M) = \rmI_1(P \cdot {\Bbb M}\cdot Q) = I + J
$$
and 
$$
\Fitt_0(F/M) = \rmI_2(\Bbb M) = \rmI_2(P \cdot {\Bbb M}\cdot Q) = IJ
$$
where $\rmI_i(\Bbb L)$ stands for the ideal of $A$ generated by $i \times i$ minors of a matrix $\Bbb L$. Note that, since $\Fitt_0(F/M)$ is an integrally closed ideal of $A$, we obtain $\overline{I}\cdot \overline{J} = I_1 {\cdot}I_2 \cdots I_{\ell}$. Thanks to Zariski's theorem (see \cite[Theorem 14.4.8]{SH}), we have $\ell = 2$, and may assume that  
$\overline{I} = I_1$, $\ord_A(I_1) \le \ord_A(I_2)$. 
Hence, $\ord_A(I) \le \ord_A(\Fitt_1(F/M))-1$, which makes a contradiction, because $I \subseteq I+J = \Fitt_1(F/M)$.
\end{proof}

We are now ready to state the example.

\begin{prop}\label{5.10}
Let $A=k[[X, Y]]$ be the formal power series ring over an infinite field $k$. We set 
$${\Bbb M} = 
\begin{pmatrix}
X^3 & X^2Y^2 & XY^3 & Y^5 & 0 & 0 & 0 \\
0 & 0 & X^3 & 0 & X^2Y^2 & XY^4 & Y^5
\end{pmatrix}
$$
and consider the presentation 
$$
A^{\oplus 7} \overset{\Bbb M}{\longrightarrow} A^{\oplus 2} \longrightarrow C \longrightarrow 0
$$
of $A$-modules.
We put $M = \Im {\Bbb M} \subseteq F = A^{\oplus 2}$. Then $\calR(M)$ is a Buchsbaum ring, $\widetilde{M\hspace{0.2em}} = \overline{M}$, $M \ne \overline{M}$, and $\calR(M)$ is not a multi-Rees algebra.
\end{prop}

\begin{proof}
One can check that $\mu_A(M) = 7$. Since $\mu_A(M) =7$, we get $M \ne \overline{M}$ by \cite[Proposition 2.2, Proposition 2.5]{K}.
Note that $\Fitt_0(F/M) = \rmI_2(\Bbb M) = (X, Y^2)\overline{(X^5, Y^8)}$ is an integrally closed ideal of $A$ and $\Fitt_1(F/M) = \rmI_1(\Bbb M) = (X^3, X^2Y^2, XY^3, Y^5)$, so that $M$ is indecomposable by Lemma \ref{5.11}.
The integral closure of $M$ is given by
$$
\overline{M} = M + \left<
\begin{pmatrix}
XY^4 \\
0
\end{pmatrix}
\right>.
$$
Indeed, we set $L=M + \left<\left(
\begin{smallmatrix}
XY^4 \\
0
\end{smallmatrix}\right)
\right>$ and write $F=At_1 + At_2$, where $t_1, t_2$ forms a free basis of $F$. Then the integral equation
$$
(XY^4t_1)^2 - Y(XY^3t_1 + X^3t_2)(XY^4t_1) + (X^3t_1)(Y^5t_2) = 0
$$
shows $XY^4t_1 \in \overline{M}$, so that $\overline{M} \supseteq L$.
Let $p_i : F \to A$ denote the $i$-th projection of $F$. We then have
$$
p_1(M) = (X, Y^2)\cdot\overline{(X^2, Y^3)} \quad \ \text{and} \quad \ p_2(M) = \overline{(X^3, Y^5)}
$$
which are integrally closed. Moreover, we have $X^4t_2 \in M$. Therefore, we get the chain of $A$-submodules of $F$;

\begin{equation*}
\begin{matrix}
J_1 = (X^3, X^2Y^2, Y^5) \\
\oplus \\
J_2=(X^4, X^2Y^2, XY^4, Y^5)
\end{matrix}
\subseteq 
M \subsetneq L \subseteq \overline{M} \subsetneq
\begin{matrix}
p_1(M) \\
\oplus \\
p_2(M)
\end{matrix}.
\end{equation*}
Note that 
$$
\ell_A(J_1 \oplus J_2 / p_1(M) \oplus p_2(M)) = 4 \ \ \text{and} \ \  \ell_A(M/ J_1 \oplus J_2) \ge 2
$$
whence we have $\overline{M} = L$, as desired.
It is straightforward to show that $\m \overline{M} \subseteq M$ and $M \overline{M} =M^2$. 
Finally, $\calR(M)$ is a Buchsbaum ring and $\widetilde{M \hspace{0.2em}} = \overline{M}$. 
\end{proof}


\medskip

Closing this paper, let us now discuss the Buchsbaum property for the fiber cone of modules.
We now define
$$
\calF(M) = A/\m \otimes_A \calR(M) \cong \calR(M)/{\m \calR(M)}
$$
and call it {\it the fiber cone of $M$}.
In 2001, J. Brennan, B. Ulrich, and W. V. Vasconcelos showed that $\dim \calF(M) = r + 1$ (see \cite[Proposition 2.2]{BUV}).
Recently, it is proved by R. Balakrishnan and A. V. Jayanthanan  an analogue of the result about the Cohen--Macaulay property for the fiber cone. More precisely, the fiber cone $\calF(M)$ is Cohen--Macaulay, if $M$ has the reduction number at most one (\cite[Theorem 1.2]{BJ}). 
By using their results, we finally reach the following.

\begin{thm}
Suppose that $\calR(M)$ is a Buchsbaum ring and $\widetilde{M\hspace{0.2em}} = \overline{M}$. Then $\calF(M)$ is a Buchsbaum ring.
\end{thm}

\begin{proof}
Thanks to Theorem \ref{5.1}, we get $C = \calR(\overline{M})/\calR(M) = \overline{M}/M$. Look at the exact sequence $0 \to \calR(M) \to \calR(\overline{M}) \to C \to 0$ of graded $\calR(M)$-modules. Applying the functor $A/\m \otimes_A -$, we have the sequenece $0 \to K \to \calF(M) \to \calF(\overline{M}) \to C \to 0$
where $K$ denotes the kernel of the induced homomorphism $\calF(M) \to \calF(\overline{M})$. Note that $K=K_1 = \m \overline{M}/\m M$. By \cite[Theorem 1.2]{BJ}, $\calF(\overline{M})$ is a Cohen--Macaulay ring. Hence, 
$\rmH^0_\fkM(\calF(M)) \cong K$, $\rmH^1_\fkM(\calF(M)) \cong C$, and $\rmH^i_\fkM(\calF(M)) = (0)$ for every $2 \le i \le r$.
Therefore, by \cite[Proposition 3.1]{G}, we conclude that $\calF(M)$ is a Buchsbaum ring.
\end{proof}


\vspace{0.5em}

\begin{ac}
The author would like to thank Shiro Goto for valuable advices and comments. 
The author is also grateful to the referee for his/her careful reading of the manuscript.
The author was partially supported by JSPS Grant-in-Aid for Young Scientists (B) 17K14176 and Waseda University Grant for Special Research Projects 2018K-444, 2018S-202.
\end{ac}


\end{document}